\documentclass[a4paper,12pt]{article}
\usepackage{latexsym}
\usepackage{amsmath}
\usepackage{amssymb}
\usepackage{enumerate}

\usepackage{theorem}
\newtheorem{theorem}{Theorem}[section]

\newtheorem{lemma}[theorem]{Lemma}
\newtheorem{corollary}[theorem]{Corollary}

\newtheorem{MT}{Main Theorem}
\newtheorem{CMT}{Corollary to Main Theorem}
\newtheorem{TCT}{A New Type of Toponogov  Comparison Theorem}

\theorembodyfont{\rmfamily}
\newtheorem{proof}{\textmd{\textit{Proof.}}}

\newtheorem{remark}[theorem]{Remark}

\newtheorem{acknowledgement}{\textmd{\textit{Acknowledgements.}}}

\makeatletter

\@addtoreset{equation}{section}
\makeatother

\newcommand{\qedd}{\hfill \Box}

\newcommand{\lra}{\longrightarrow}
\newcommand{\wt}{\widetilde}

\newcommand{\ol}{\overline}

\newcommand{\B}{\ensuremath{\mathbb{B}}}

\newcommand{\R}{\ensuremath{\mathbb{R}}}
\newcommand{\Sph}{\ensuremath{\mathbb{S}}}

\def\vol{\mathop{\mathrm{vol}}\nolimits}

\def\Ric{\mathop{\mathrm{Ric}}\nolimits}

\title
{
Sufficient Conditions for Open Manifolds 
to be Diffeomorphic to Euclidean Spaces\footnote{
2010 Mathematics Subject Classification. Primary 53C21; Secondary 53C22.}
\footnote{
Keywords\,:\,volume growth, radial curvature, Ricci curvature
}
}
\author{Kei KONDO $\cdot$ Minoru TANAKA}
\date{}
\begin{document}
\maketitle
\begin{abstract} 
Let $M$ be a complete 
non-compact connected Riemannian $n$-dimensional manifold. 
We first prove that, for any fixed point $p \in M$, the radial Ricci curvature of $M$ at $p$ is 
bounded from below by the radial curvature function of some non-compact $n$-dimensional model. 
Moreover, we then prove, without the pointed Gromov--Hausdorff convergence theory, 
that, if {\bf model} volume growth is sufficiently close to $1$, 
then $M$ is diffeomorphic to Euclidean $n$-dimensional space. Hence, 
our main theorem has various advantages of the Cheeger--Colding diffeomorphism 
theorem via the {\bf Euclidean} volume growth. 
Our main theorem also contains a result of do Carmo and Changyu as a special case. 
\end{abstract}

\section{Introduction}\label{sec:intro}
In the geodesic theory of global Riemannian geometry, 
the critical point theory of distance functions, introduced by Grove and Shiohama \cite{GS}, 
provides a useful application to study the relationship between the topology and geometry of 
a given Riemannian manifold. Here we say that 
a point $q$ in a complete Riemannian manifold $M$ is a {\em critical point of  
the distance function $d(p, \, \cdot \, )$ to $p \in M$} (or a {\it critical point $q$ for $p$}), 
if for every nonzero tangent vector $v$ in the tangent space $T_{q} M$ to $q$, 
there exists a minimal geodesic segment $\gamma$ emanating from $q$ to $p$ satisfying 
$\angle(v, \gamma'(0)) \le \pi / 2$, where $\angle(v, \gamma'(0))$ denotes 
the angle between two vectors $v$ and $\gamma'(0)$ in $T_{q} M$.\par  
For complete non-compact Riemannian manifolds with bounded sectional curvature, 
this critical point theory becomes particularly useful when used in conjunction with 
Toponogov's comparison theorem. It is possible to investigate whether $M$ has critical points or not 
by using the technique of drawing a circle or a geodesic polygon, joining two points 
by a minimal geodesic segment, and finally estimating the angles of geodesic triangles on $M$. 
If $M$ admits a region which has no critical points, then the shape of the region 
can be stretched and deformed into a region on a plane 
(cf.\,\cite{GS}, Corollary 1.4 in \cite[Chapter 11]{P}). In particular 
$M$ is diffeomorphic to Euclidean $n$-dimensional space $\R^{n}$, if 
$M$ does not have any critical points of $d(p, \, \cdot \, )$ for a fixed point $p \in M$.

\bigskip

To control a set of critical points of the distance function on a non-compact Riemannian $n$-dimensional manifold $M$ with {\bf non-negative} Ricci curvature {\bf everywhere}, 
Otsu \cite{O} very first introduced the Euclidean volume growth 
\begin{equation}\label{EVG}
\lim_{t \to \infty} \frac{\vol B_{t} (x)}{t^{n} \vol \Sph^{n -1}(1)},  
\end{equation} 
where $\vol B_{t}(x)$ denotes the volume of the open distance ball $B_{t}(x)$ 
at a point $x \in M$ with radius $t > 0$ in $M$, and $\vol \Sph^{n -1}(1)$ denotes the volume of 
the unit ball $\Sph^{n -1}(1)$ in Euclidean $n$-dimensional space $\R^{n}$. 
Notice that, by the Bishop volume comparison theorem, 
\[
\lim_{t \to \infty} \frac{\vol B_{t} (x)}{t^{n} \vol \Sph^{n -1}(1)} \le 1.
\] 
If (\ref{EVG}) equals $1$, the $M$ is isometric to $\R^{n}$. 
Hence, it is very natural to expect $M$ to be diffeomorphic to $\R^{n}$, 
when (\ref{EVG}) is sufficiently close to $1$. In fact, Otsu proved 

\begin{theorem}{\rm (\cite[Theorem 1.2]{O})}\label{thm1.1-2009-11-05}
Let $M$ be a complete non-compact Riemannian $n$-manifold 
with non-negative Ricci curvature, 
and let $\lambda : [0, \infty) \lra \R$ be a negative increasing continuous function 
such that 
\begin{enumerate}[{\rm ({O--}1)}]
\item
$\displaystyle{
c_{0} := \int_{0}^{\infty} t \lambda (t) \,dt > -\infty
}$ and that 
\item
the sectional curvature at any point $q \in M$ is bounded from below by $\lambda (d (p, q))$ 
for some fixed point $p \in M$
\end{enumerate}
Then, there exists $\delta (n, c_{0}) > 0$ such that, if 
\[
\lim_{t \to \infty} \frac{\vol B_{t} (x)}{t^{n} \vol \Sph^{n -1}(1)} 
\ge 1 - \delta (n, c_{0})
\]
for some $x \in M$, 
then $M$ is diffeomorphic to Euclidean $n$-space $\R^{n}$. 
\end{theorem}
 
\medskip\noindent
Notice that (O--1) and (O--2) imply that the manifold $M$ is at least as curved as 
a model surface of revolution with a finite total curvature. 

\bigskip

There is a {\bf great} number of related results for Theorem \ref{thm1.1-2009-11-05}. 
However, after Colding's study of the relationship between Ricci curvatures on complete Riemannian manifolds, Gromov--Hausdorff convergence theory and volumes of the manifolds (\cite{C}), 
Cheeger and Colding proved the next theorem, 
which shines out very much among such related results: 

\begin{theorem}{\rm (\cite[Theorem A.1.11]{CC})}\label{thm1.2-2009-11-05}
Let $M$ be a complete non-compact Riemannian $n$-manifold with non-negative Ricci curvature. 
Then, there exists $\delta (n) > 0$ such that, if 
\[
\vol B_{t} (x) \ge (1 - \delta (n)) \vol \Sph^{n -1}(1) t^{n}
\]
for all $x \in M$, $t > 0$, then $M$ is diffeomorphic to Euclidean $n$-space $\R^{n}$. 
\end{theorem}

\bigskip

Our purpose of this article is to extend Theorem \ref{thm1.2-2009-11-05} to any complete 
non-compact connected Riemannian manifold $M$, i.e., we will remove the 
non-negative Ricci curvature condition in Theorem \ref{thm1.2-2009-11-05}. 
To state that precisely, we will begin by defining the radial curvature geometry. 

\bigskip

Let $\wt{M}^{n}$ denote a complete non-compact Riemannian $n$-dimensional 
manifold, which is homeomorphic to $\R^{n}$, with a base point $\tilde{p} \in \wt{M}^{n}$. 
Then, we call the pair $(\wt{M}^{n}, \tilde{p})$ an $n$-dimensional {\em model} 
if its Riemannian metric $d\tilde{s}^2$ is expressed 
in terms of geodesic polar coordinates around $\tilde{p}$ as 
\begin{equation}\label{polar}
d\tilde{s}^2 = dt^2 + f(t)^2d \theta^2, \quad 
(t,\theta) \in (0,\infty) \times \Sph_{\tilde{p}}^{n -1}.
\end{equation}
Here $f : (0, \infty) \lra \R$ is a positive smooth function 
which is extendible to a smooth odd function around $0$, 
and $d \theta$ denotes the Riemannian metric on the unit sphere 
$\Sph^{n - 1}_{\tilde{p}} := \{ v \in T_{\tilde{p}} \wt{M}^{n} \ | \ \| v \| = 1 \}$. 
The function $G \circ \tilde{\gamma} : [0,\infty) \lra \R$ is called the 
{\em radial curvature function} of $(\wt{M}^{n}, \tilde{p})$, 
where we denote by $G$ the sectional curvature of $\wt{M}^{n}$, 
and by $\tilde{\gamma}$ any meridian emanating from 
$\tilde{p} = \tilde{\gamma} (0)$. 
Note that $f$ satisfies the differential equation 
\[
f''(t) + G (\tilde{\gamma}(t)) f(t) = 0
\]
with initial conditions $f(0) = 0$ and $f'(0) = 1$. 
The $n$-models are completely classified in \cite{KK}. 
In particular, if $n = 2$, a model are called a {\em non-compact model surface of revolution}.\par  
Let $(M,p)$ be a complete non-compact Riemannian 
$n$-dimensional manifold with a base point $p \in M$.
We say that $(M, p)$ has {\em radial Ricci curvature at $p$ bounded 
from below by the radial curvature function of 
an $n$-model $(\wt{M}^{n}, \tilde{p})$} 
if,
along every unit speed minimal geodesic $\gamma: [0,a) \lra M$ emanating from 
$\gamma (0) = p$, its Ricci curvature $\Ric_p$ with respect to $\gamma'(t)$ satisfies
\[
\Ric_p (\gamma'(t)) := \frac{1}{n -1}
\sum_{i = 1}^{n - 1} \langle R(\gamma'(t), e_{i})\gamma'(t), e_i \rangle \ge G (\tilde{\gamma}(t))
\]
for all $t \in [0,a)$. Here 
$R$ denotes the Riemannian curvature tensor of $M$, which is a multi-linear map, 
defined by 
$R(X, Y)Z := \nabla_{Y}\nabla_{X}Z -  \nabla_{X}\nabla_{Y}Z + \nabla_{[X, Y]}Z$ 
for smooth vector fields $X, Y, Z$ over $M$ and 
$\{ e_{1}, e_{2}, \cdots, e_{n - 1} \} := \{ e_{1}(t), e_{2}(t), \cdots, e_{n - 1}(t) \}$ 
denotes an orthonormal basis of the hyperplane in  $T_{\gamma(t)}M$ orthogonal to 
$\gamma'(t)$. For example, if the Riemannian metric of $\wt{M}$ is 
$dt^2 + t^{2}d \theta^2$, or $dt^2 + \sinh^{2} t\,d \theta^2$, then 
$G (\tilde{\gamma}(t)) = 0$, or $G (\tilde{\gamma}(t)) = -1$, respectively. 
Notice that {\bf the radial Ricci curvature may change signs wildly}. 
For example, there exist model surfaces of revolution 
with finite total curvature whose Gauss curvatures are not  bounded, i.e., 
such surfaces satisfy $\liminf_{t \to \infty} G (\tilde{\gamma}(t)) = - \infty$, or 
$\limsup_{t \to \infty} G (\tilde{\gamma}(t))= \infty$ (see \cite[Theorems 1.3 and 4.1]{KT4}).

\bigskip

To state our main theorem, we need to introduce an essential ratio and an important function: 
Let $(M,p)$ be a complete non-compact connected Riemannian 
$n$-dimensional manifold $M$ whose radial Ricci curvature at the base point $p$ 
is bounded from below by the radial curvature function of an 
$n$-dimensional model $(\wt{M}^{n}, \tilde{p})$. Under this curvature relationship between 
$M$ and $\wt{M}^n$, the limit 
\[
\lim_{t \to \infty} \frac{\vol B_{t} (p)}{\vol B_{t} (\tilde{p})} 
\]
is called the {\em model volume growth}, where 
$B_{t}(p) \subset M$ denotes the open distance ball at $p$ with radius $t > 0$, 
and $B_{t} (\tilde{p}) \subset \wt{M}^n$ denotes the open distance ball at $\tilde{p}$ 
with radius $t > 0$. Furthermore we define a function
\begin{equation}\label{net-function}
F(r) := \left( \int_{0}^{\pi} \sin^{n -2} t\,dt \right)^{-1} \int_{0}^{r} \sin^{n -2} t\,dt
\end{equation}
on $[0, \pi]$, and we call it {\em the net function for} 
$\Sph^{n - 1}_{p} := \{ v \in T_{p} M \ | \ \| v \| = 1 \}$.\par 
Now our main theorem is stated as follows, 
which has various advantages of the Cheeger--Colding theorem above:

\medskip

\begin{MT} 
Let $M$ be a complete non-compact connected Riemannian $n$-manifold. Then, for any fixed point 
$p \in M$, 
\begin{enumerate}[{\rm ({A--}1)}]
\item 
There exists a locally Lipschitz function $G(t)$ (respectively $K(t)$) on $[0, \infty)$ 
such that radial Ricci (respectively sectional) curvature of $M$ at $p$ 
is bounded from below by that of an $n$-model $(\wt{M}^{n}, \tilde{p})$ 
with $G$ (respectively that of a non-compact model surface of revolution with $K$) 
as its radial curvature function.
\item
Moreover, if 
\begin{enumerate}[{\rm ({B--}1)}]
\item
$
\displaystyle{
\lim_{t \to \infty} \vol B_{t}(\tilde{p}) = \infty
}$ and  
\item
$
\displaystyle{
\lim_{t \to \infty} \frac{\vol B_{t} (p)}{\vol B_{t}(\tilde{p})} 
\ge 1 - F(\delta (K^{*}))
}$,
\end{enumerate}
then $M$ is diffeomorphic to Euclidean $n$-space $\R^{n}$.
Here $B_{t}(p) \subset M$ and $B_{t} (\tilde{p}) \subset \wt{M}^n$ 
denote the open distance balls at $p$ and at $\tilde{p}$ with radius $t > 0$, respectively, 
and we set $K^{*} := \min \{0, G, K\}$
and 
\[
\delta (K^{*}) := \frac{\pi}{2} \exp \left( \int^{\infty}_{0} t \,K^{*} (t) dt \right).
\]
\end{enumerate}
\end{MT}

\bigskip\noindent
Here, we say that $M$ has {\em 
radial sectional curvature at the base point $p \in M$ bounded 
from below by that of a non-compact model surface of revolution} 
if, along every unit speed minimal geodesic $\gamma: [0,a) \lra M$ 
emanating from $p = \gamma (0)$, 
its sectional curvature $K_M(\sigma_{t})$ is bounded from below by the radial curvature function 
of the surface for all $t \in [0, a)$ and all $2$-dimensional linear spaces 
$\sigma_{t}$ spanned by $\gamma'(t)$ and a tangent vector to $M$ at $\gamma(t)$.\par 
The first assertion (A--1) is already proved for the radial sectional curvature of the manifold 
at any fixed point (see \cite[Lemma 5.1]{KT2}).\par 
In the second assertion (A--2), it is not necessary, as a condition, 
whether the value $\int^{\infty}_{0} t \,K^{*} (t) dt$ is finite or not. 
Moreover, the (A--2) has at least two advantages of Theorem \ref{thm1.2-2009-11-05}, 
which are as follows\,: 
\begin{enumerate}[{\rm (1)}]
\item
The condition (B--1) is natural, because we may easily find 
such a $(\wt{M}^{n}, \tilde{p})$. For example, $\wt{M}^{n} = \R^{n}$. 
\item
Our volume growth is bounded from below by a {\bf definite} constant.
\end{enumerate}

\medskip

We may remove the condition (B--1) by assuming 
the radial sectional curvature is bounded below: 

\begin{CMT}
Let $(M,p)$ be a complete non-compact Riemannian $n$-manifold $M$ 
whose radial sectional curvature at the base point $p$ is bounded from below by
the radial curvature function $G$ of a non-compact 
model surface of revolution $(\wt{M}, \tilde{p})$. If 
\[
\lim_{t \to \infty} \frac{\vol B_{t} (p)}{\vol B_{t}^{n} (\tilde{p})} 
\ge 1 - F(\delta (G_{-}))
\]
then $M$ is diffeomorphic to Euclidean $n$-space $\R^{n}$.
Here $B_{t}^{n} (\tilde{p})$ denotes 
the open distance ball at $\tilde{p} \in \wt{M}^n$ with radius $t > 0$ in 
an $n$-dimensional model $(\wt{M}^n, \tilde{p})$, 
and we set $G_{-} := \min \{0, G\}$.
\end{CMT}

\medskip\noindent
Hence, by our main theorem and this corollary, we realize that 
the difference between Ricci curvature and sectional curvature is whether the volume of each comparison model is finite or not. 
Notice that this corollary directly contains a results of 
do Carmo and Changyu (\cite{CaCh}) as a special case, that is, $f(t)=t$, where 
$f$ is the warping function of the surface $(\wt{M}, \tilde{p})$. 

\medskip

In the following sections, 
all geodesics will be normalized, unless otherwise stated.  

\begin{acknowledgement}
The first named author would like to express to Professors J. Dodziuk, C. Sormani, N. Katz, and 
D. Lee his deepest gratitude for their helpful comments on the first version of our main theorem in 
the differential geometry seminar at the CUNY graduate center, New York City, 
8th and 15th September, 2009. 
\end{acknowledgement}
\section{Mass of Rays and Volume Growth}\label{sec:mass}

The purpose of this section is to investigate the relationship between mass of rays and the model 
volume growth. Especially, Lemma \ref{lem2009-08-29-2} is 
the key lemma to prove our main theorem. Since this lemma was stated in \cite{O} without a proof, 
we will give a proof of it here. 

\bigskip

Throughout this section, let $(M,p)$ denote a complete non-compact Riemannian 
$n$-dimensional manifold $M$ 
whose radial Ricci curvature at the base point $p$ is bounded from below by
the radial curvature function $G (\tilde{\gamma}(t))$ of an $n$-dimensional 
model $(\wt{M}^{n}, \tilde{p})$ with its metric (\ref{polar}).
Let $A_{p}$ be the set of all unit vectors tangent to rays emanating from $p \in M$. 
Then, it is clear that  
$A_{p} = \left\{ v \in \Sph^{n - 1}_{p} \ | \ \rho (v) = \infty \right\}$. 
Here we set 
\[
\rho(v) := \sup \{ t > 0 \ | \ d(p, \gamma_{v}(t)) = t \},
\]
where 
$\gamma_{v}$ denotes the unit speed geodesic emanating from $p \in M$ such that 
$v = \gamma'_{v}(0) \in \Sph^{n - 1}_{p}$. 
Let $\{e_{1}, e_{2}, \ldots, e_{n} \}$
be an orthonormal basis of $T_{p}M$ such that  $e_{n} := v \in \Sph^{n - 1}_{p}$. 
Take Jacobi fields $Y_{i}(t, v)$ along the unit speed geodesic $\gamma_{v}$ 
emanating from $p \in M$ such that 
\[
Y_{i} (0, v) = 0, \quad Y'_{i} (0, v) = e_{i}, \quad i = 1,2, \ldots, n-1.
\]
Here $Y'_{i}$ denotes the covariant derivative of $Y_{i}$ along $\gamma_{v}$. 
Then, we set
\[
\Theta (t, v) := \sqrt{\det \left( \left\langle Y_{i}(t, v), Y_{j}(t, v) \right\rangle \right)}, \quad 
1 \le i,\, j \le n - 1.
\]
We define 
\[
\ol{\Theta} (t, v) = 
\left\{
\begin{array}{ll}
\Theta (t, v),& \qquad t \le \rho(v), \\[2mm]
0, & \qquad t > \rho(v).
\end{array}
\right.
\]
Then, 
\[
\vol B_{t} (p) = \int_{0}^{t} dr \int_{\Sph^{n - 1}_{p}} \ol{\Theta} (r, v) \,d\Sph^{n - 1}_{p}.
\]
As well as above, for $(\wt{M}^{n}, \tilde{p})$, 
we may consider the corresponding notions $\Sph^{n - 1}_{\tilde{p}}$, $\tilde{\gamma}_{\tilde{v}}$, $\wt{Y}_{i} (t, \tilde{v})$, $\wt{\Theta} (t, \tilde{v})$, etc. 
Since $\wt{\Theta} (t, \tilde{v}) = f(t)^{n - 1}$, we have 
\begin{equation}\label{model-vol}
\vol B_{t} (\tilde{p}) = \omega_{n - 1} \int_{0}^{t} f(r)^{n - 1}\,dr,
\end{equation}
where we set $\omega_{n - 1} := \vol \Sph^{n - 1}_{\tilde{p}}$.

\begin{lemma}\label{lem4.1}
If 
${\lim_{t \to \infty} \vol B_{t} (\tilde{p}) = \infty}$,
then
\[
\vol_{\Sph_{p}^{n - 1}} A_{p} \ge \omega_{n - 1} 
\lim_{t \to \infty} \frac{\vol B_{t} (p)}{\vol B_{t} (\tilde{p})}.
\]
\end{lemma}

\begin{proof}
Let $U(A_{p})$ denote any open neighborhood of $A_{p}$ in $\Sph^{n - 1}_{p}$. 
Since $M$ is complete and non-compact, 
there exists $t_{0} > 0$ such that, for any minimal geodesic segment 
$\gamma|_{[0, \,t_{0}]}$ emanating from $p$, 
$\gamma'(0) \in U(A_{p})$. 
It follows from the Bishop volume comparison theorem and (\ref{model-vol}) that, 
for any $t > t_{0}$, 
\begin{align}\label{lem4.1-1}
\vol B_{t} (p) 
&\le \int_{0}^{t_{0}} dr \int_{\Sph^{n - 1}_{p}} \ol{\Theta} (r, v) \,d\Sph^{n - 1}_{p}
+ \int_{t_{0}}^{t} dr \int_{U(A_{p})} \ol{\Theta} (r, v) \,d\Sph^{n - 1}_{p}\notag\\[2mm]
&\le \int_{0}^{t_{0}} dr \int_{\Sph^{n - 1}_{p}} \ol{\Theta} (r, v) \,d\Sph^{n - 1}_{p}
+ \int_{t_{0}}^{t} dr \int_{U (A_{p})} \wt{\Theta} (r, v) \,d\Sph^{n - 1}_{p}\notag\\[2mm]
&= \int_{0}^{t_{0}} dr \int_{\Sph^{n - 1}_{p}} \ol{\Theta} (r, v) \,d\Sph^{n - 1}_{p}
+ \vol_{\Sph_{p}^{n - 1}}U(A_{p}) \int_{t_{0}}^{t} f(r)^{n - 1}\,dr\notag\\[2mm]
&= \int_{0}^{t_{0}} dr \int_{\Sph^{n - 1}_{p}} \ol{\Theta} (r, v) \,d\Sph^{n - 1}_{p}\notag\\[2mm]
& \quad 
+ \frac{\vol_{\Sph_{p}^{n - 1}}U(A_{p})}{\omega_{n -1}}
\left( 
\vol B_{t} (\tilde{p}) - \vol B_{t_{0}}(\tilde{p})
\right).
\end{align}
Then, by the equation (\ref{lem4.1-1}), we have 
\[
\frac{\vol B_{t} (p)}{\vol B_{t}(\tilde{p})} 
\le
\frac{\int_{0}^{t_{0}} dr \int_{\Sph^{n - 1}_{p}} \ol{\Theta} (r, v) \,d\Sph^{n - 1}_{p}}
{\vol B_{t} (\tilde{p})} 
+ \frac{\vol_{\Sph_{p}^{n - 1}}U(A_{p})}{\omega_{n -1}}
\left( 
1 - \frac{\vol B_{t_{0}}(\tilde{p})}{\vol B_{t} (\tilde{p})}
\right).
\]
Thus, we see 
\[
\lim_{t \to \infty} \frac{\vol B_{t} (p)}{\vol B_{t}(\tilde{p})} 
\le 
\frac{\vol_{\Sph_{p}^{n - 1}}U (A_{p})}{\omega_{n -1}},
\]
i.e., 
\[
\vol_{\Sph_{p}^{n - 1}}U (A_{p}) \ge \omega_{n - 1} 
\lim_{t \to \infty} \frac{\vol B_{t} (p)}{\vol B_{t}(\tilde{p})}.
\]
Since $U (A_{p})$ is arbitrary, 
we hence get 
\[
\vol_{\Sph_{p}^{n - 1}}A_{p} \ge \omega_{n - 1} 
\lim_{t \to \infty} \frac{\vol B_{t} (p)}{\vol B_{t}(\tilde{p})}.
\] 
$\qedd$
\end{proof}

\bigskip

Let $F$ denote the net function for $\Sph^{n - 1}_{p}$ on $[0, \pi]$ 
(see (\ref{net-function}) for its definition). 

\begin{lemma}\label{lem2009-08-29-1}
Let $\B_{\delta}(v) \subset \Sph^{n - 1}_{p}$ denote 
the open ball centered at $v \in \Sph^{n - 1}_{p}$ with radius $\delta \in [0, \pi]$. 
Then, 
$\vol \B_{\delta}(v) = \omega_{n -1} F(\delta)$
for all $\delta \in [0, \pi]$. 
\end{lemma}

\begin{proof}
This is clear, since 
$
\vol \B_{\delta}(v) = \omega_{n -2} \int_{0}^{\delta} \sin^{n-2} t \,dt
$
holds for all $\delta \in [0, \pi]$.
$\qedd$
\end{proof}

\begin{lemma}\label{lem2009-08-29-2} 
Let $\delta$ be a constant number in $[0, \pi]$. 
If $\lim_{t \to \infty} \vol B_{t} (\tilde{p}) = \infty$ and 
\begin{equation}\label{lem2009-08-29-2-A}
\lim_{t \to \infty} \frac{\vol B_{t} (p)}{\vol B_{t}(\tilde{p})} 
\ge 
1 - F(\delta),
\end{equation}
then
$
\vol_{\Sph_{p}^{n - 1}} A_{p} \ge \vol \B_{\pi - \delta}(v)
$
holds for all $v \in \Sph^{n - 1}_{p}$. 
\end{lemma}

\begin{proof}
By Lemma \ref{lem4.1}, Lemma \ref{lem2009-08-29-1}, and (\ref{lem2009-08-29-2-A}), 
$\vol_{\Sph^{n -1}_{p}} A_{p} 
\ge 
\omega_{n - 1} - \vol \B_{\delta} (v)
$ 
holds for all $v \in \Sph^{n - 1}_{p}$. 
Hence, we get the assertion. 
$\qedd$
\end{proof}

\section{Proofs of Diffeomorphism Theorems}\label{sec:proof}

The purpose of this section is to prove 
Main Theorem (Theorem \ref{prop3.2-2009-11-04} and \ref{prop2009-09-29}) 
and its corollary (Corollary \ref{cor3.6-2009-11-05}). 
Throughout this section, let $M$ denote a complete non-compact connected 
Riemannian $n$-dimensional manifold.

\begin{theorem}\label{prop3.2-2009-11-04} 
For any fixed point $p \in M$, 
there exist locally Lipschitz functions $G(t)$ (respectively $K(t)$) on $[0, \infty)$ 
such that radial Ricci (respectively sectional) curvature of $(M, p)$ at $p$ 
is bounded from below by that of an $n$-model with $G$ (respectively that of a 
non-compact model surface of revolution with $K$) as its radial curvature function.
\end{theorem}

\begin{proof}
We will state the outline of the proof, since the proof is the very same as that of \cite[Lemma 5.1]{KT2}. 
Let $\gamma_{v} : [0, \rho(v)] \lra M$ denote a minimal geodesic emanating from 
$p = \gamma_{v}(0)$ such that $v = \gamma'_{v}(0) \in \Sph^{n - 1}_{p}$, 
where $\rho(v) := \sup \{ t > 0 \ | \ d(p, \gamma_{v}(t)) = t \}$. 
For each $v \in \Sph^{n - 1}_{p}$, let $\Ric_p (\gamma_{v}'(t))$ be the radial Ricci curvature of $M$ 
at $p$ along $\gamma_{v}$. Now, we define a function $G$ on $[0, \infty)$ by 
$G(t) := \min \left\{\Ric_p (\gamma_{v}'(\rho_{t}(v))) \ | \ v \in \Sph^{n - 1}_{p}\right\}$ 
where $\rho_{t}(v) := \min \{\rho (v), t\}$. 
It is easy to check that $G(t)$ has the required properties.\par 
For a locally Lipschitz function $K(t)$ on $[0, \infty)$ which  
bounds the radial sectional curvature of $M$ at $p$ from below, see \cite[Lemma 5.1]{KT2}. 
$\qedd$
\end{proof}

By Theorem \ref{prop3.2-2009-11-04}, 
we may apply a new type of the Toponogov comparison theorem to the pair $(M, p)$ 
in Theorem \ref{prop3.2-2009-11-04}, which was established by the present authors 
as generalization of the comparison theorem in conventional comparison geometry:

\medskip

\begin{TCT}\label{TCT}\hspace{-1.5mm}{\rm (\cite[Theorem 4.12]{KT2})}\ \par
Let $(X,o)$ be a complete non-compact Riemannian manifold $X$ 
whose radial sectional curvature at the base point $o$ is bounded from below by
that of a non-compact model surface of revolution $(\wt{X}, \tilde{o})$ 
with its metric $dt^2 + h(t)^2d \theta^2$, $(t,\theta) \in (0,\infty) \times \Sph_{\tilde{o}}^1$. 
If $(\wt{X}, \tilde{o})$ admits a sector 
\[
\wt{V}(\delta_{0}) := \{ \tilde{x} \in \wt{X} \, | \, 0 < \theta(\tilde{x}) < \delta_{0} \}, \quad 
\delta_{0} \in (0, \pi],
\]
having no pair of cut points, 
then, for every geodesic triangle $\triangle(oxy)$ in $(X,o)$ 
with $\angle (xoy) < \delta_{0}$, 
there exists a geodesic triangle 
$\wt{\triangle} (oxy) :=\triangle(\tilde{o}\tilde{x}\tilde{y})$ 
in $\wt{V}(\delta_{0})$ such that
\begin{equation}\label{TCT-length}
d(\tilde{o},\tilde{x})=d(o,x), \quad d(\tilde{o},\tilde{y})=d(o,y), \quad d(\tilde{x},\tilde{y})=d(x,y) 
\end{equation}
and that
\[
\angle (xoy) \ge \angle (\tilde{x}\tilde{o}\tilde{y}), \quad  
\angle (oxy) \ge \angle (\tilde{o}\tilde{x}\tilde{y}), \quad
\angle (oyx) \ge \angle (\tilde{o}\tilde{y}\tilde{x}). 
\]
Here $\angle(oxy)$ denotes the angle between the minimal geodesic segments 
from $x$ to $o$ and $y$ forming the triangle $\triangle(oxy)$.
\end{TCT}

\bigskip\noindent
Notice that the assumption on $\wt{V}(\delta_{0})$ in our comparison theorem 
is automatically satisfied, if we employ a von Mangoldt surface of 
revolution (which is, by definition, its radial curvature function is non-increasing on $[0, \infty)$), 
or a Cartan--Hadamard surface of revolution (which is, by definition, its radial curvature function is non-positive on $[0, \infty)$) as a $(\wt{X}, \tilde{o})$ for 
$\delta_{0} \le \pi$.

\begin{remark}
In \cite{KT3}, the present authors very recently generalized, from the radial curvature 
geometry's standpoint, the Toponogov comparison theorem to a complete Riemannian manifold 
with smooth convex boundary.
\end{remark}

\medskip

By the same argument in the proof of \cite[Theorem 5.3]{KT2}, we have 

\begin{lemma}{\rm (see \cite[Theorem 5.3]{KT2})}\label{lem3.1}
Let $(M^{*}, p^{*})$ be 
a non-compact model surface of revolution with its metric 
$dt^2 +  m(t)^2d \theta^2$, $(t,\theta) \in (0,\infty) \times \Sph_{p^{*}}^1$,
satisfying the differential equation
$m''(t) + K (t) m(t) = 0$ with $m(0) = 0$ and $m'(0) = 1$. 
Here $K :[0, \infty) \lra \R$ denotes a continuous function. 
If $M^{*}$ satisfies  
\[
\int^{\infty}_{0} t \,K (t) \,dt > - \infty
\]
and $K(t) \le 0$ on $[0, \infty)$, then 
\[
1 \le \lim_{t \to \infty} m'(t) \le \exp \left( \int^{\infty}_{0} (-t \,K (t)) \,dt \right) < \infty
\]
holds. 
In particular, $M^{*}$ admits a finite total curvature. 
\end{lemma}

\medskip

Take any $p \in M$, and fix it. From now on, for the $p$, 
let $G, K$ be locally Lipschitz functions on $[0, \infty)$ in 
Theorem \ref{prop3.2-2009-11-04}, respectively. 
Let $(\wt{M}^{n}, \tilde{p})$ denote an $n$-model with the $G$ as its radial curvature function, i.e., 
\[
\Ric_p (\gamma_{v}'(t)) \ge G(\tilde{\gamma}(t))
\] 
on $[0, \infty)$, and let $B_{t}(p)$ (respectively $B_{t} (\tilde{p})$) denote 
the open distance ball at $p$ 
with radius $t > 0$ in $M$ (respectively the open distance ball at $\tilde{p} \in \wt{M}^n$ with radius $t > 0$ in $\wt{M}^n$). Moreover, we denote by $(M^{*}, p^{*})$ 
a non-compact model surface of revolution with its metric 
$
g^{*} = dt^2 +  m(t)^2d \theta^2$, 
$(t,\theta) \in (0,\infty) \times \Sph_{p^{*}}^1$, 
satisfying the differential equation
\[
m''(t) + K^{*} (t) m(t) = 0
\]
with $m(0) = 0$ and $m'(0) = 1$, where $K^{*} := \min \{0, G, K\}$. 
Notice that we may take $(M^{*}, p^{*})$ a comparison surface for the  pair $(M, p)$ whenever 
we apply a new type of the Toponogov  comparison theorem to $(M, p)$, 
since $K(t) \ge K^{*} (t)$ and $K^{*} (t) \le 0$ on $[0, \infty)$. 

\bigskip

\begin{theorem}\label{prop2009-09-29}
If $\lim_{t \to \infty} \vol B_{t} (\tilde{p}) = \infty$ and 
\begin{equation}\label{prop2009-09-29-A}
\lim_{t \to \infty} \frac{\vol B_{t} (p)}{\vol B_{t} (\tilde{p})} 
\ge 1 - F(\delta (K^{*}))
\end{equation}
then $M$ is diffeomorphic to Euclidean $n$-space $\R^{n}$. 
Here $F$ denotes the net function for $\Sph^{n -1}_{p}$, and we set 
\[
\delta(K^{*}) := \frac{\pi}{2} \exp \left( \int^{\infty}_{0} t \,K^{*} (t) \, dt \right).
\]
\end{theorem}

\begin{proof}
We first consider the case where 
\[
\int^{\infty}_{0} t \,K^{*} (t) \, dt = -\infty.
\]
Then, since 
\[
\lim_{t \to \infty} \frac{\vol B_{t} (p)}{\vol B_{t} (\tilde{p})} 
= 1
\]
holds, $M$ is isometric to $\wt{M}^{n}$. Hence, $M$ is diffeomorphic to $\R^{n}$.\par
Next, we consider the case where 
\begin{equation}\label{prop2009-09-29-1}
\int^{\infty}_{0} t \,K^{*} (t) \, dt > -\infty.
\end{equation}
From the critical point theory 
(cf.\,\cite{GS}, Corollary 1.4 in \cite[Chapter 11]{P}), 
it is sufficient to prove that any point distinct from $p$ is not critical of $d(p, \, \cdot \, )$. 
Suppose that there exists a critical point 
$x \in M \setminus \{ p \}$ of $d(p, \, \cdot \, )$. 
Let $\gamma : [0, d(p, x)] \lra M$ be any minimal geodesic segment joining from $p = \gamma (0)$ 
to $x = \gamma (d(p, x))$, and let $\mu : [0, \infty) \lra M$ be any ray emanating from $p = \mu (0)$. 
By the Cohn\,-\,Vossen's technique 
(see \cite{CV2}, or \cite[Lemma 2.2.1]{SST}), 
there exist a divergent sequence $\{ t_{i} \}$ and a sequence of minimal geodesic segments 
$\eta_{i} : [0, \ell_{i}] \lra M$ emanating from $x = \eta_{i} (0)$ to $\mu (t_{i}) = \eta_{i} (\ell_{i})$, 
where $\ell_{i} := d(x, \mu(t_{i}))$, such that  
\begin{equation}\label{thm5.1-5}
\lim_{i \to \infty} \angle (\eta_{i}'(\ell_{i}), \mu'(t_{i})) = 0.
\end{equation}
Since $x$ is a critical point of $d(p, \,\cdot\,)$, 
for each $\eta_{i}$, there exists a minimal geodesic segment 
$\sigma_{i} : [0, d(p, x)] \lra M$ emanating from $x$ to $p$ 
such that 
\begin{equation}\label{2009-03-11-2}
\angle (\sigma_{i}' (0), \eta_{i}'(0)) \le \pi / 2.
\end{equation}
Then, it follows from a new type of the Toponogov comparison theorem that 
there exists a geodesic triangle 
$
\triangle( p^{*} x^{*} \mu(t_{i})^{*} ) \subset M^{*}
$ 
corresponding to the triangle 
$\triangle (p x \mu(t_{i})) \subset M$ 
which consists of the sides $\gamma$, $\eta_{i}$, and $\mu|_{[0, \, t_{i}]}$ 
such that (\ref{TCT-length}) holds (for $o = p$ and $y = \mu (t_{i})$) and that 
\begin{equation}\label{2009-03-12-1}
\angle (x^{*} p^{*} \mu(t_{i})^{*})
\le \angle (\gamma'(0), \mu'(0)),
\end{equation}
\begin{equation}\label{2009-03-12-2}
\angle (p^{*} \mu(t_{i})^{*} x^{*})
\le
\angle (p \mu(t_{i}) x).
\end{equation}
By (\ref{thm5.1-5}) and (\ref{2009-03-12-2}), 
\begin{equation}\label{2009-03-12-3}
\lim_{i \to \infty} \angle (p^{*} \mu(t_{i})^{*} x^{*}) = 0.
\end{equation}
On the other hand, we denote by 
$\triangle (p \sigma_{i} (0) \mu(t_{i})) \subset M$ 
the geodesic triangle consisting of the sides $\sigma_{i}$, $\eta_{i}$, 
and $\mu|_{[0, \, t_{i}]}$. 
By our Toponogov comparison theorem and (\ref{2009-03-11-2}), 
we have 
\begin{equation}\label{thm5.1-14}
\angle (p^{*} x^{*} \mu(t_{i})^{*}) \le \pi / 2.
\end{equation}
Applying the Gauss\,--\,Bonnet Theorem to 
the geodesic triangle $\triangle (p^{*} x^{*} \mu(t_{i})^{*})$, 
we have 
\begin{align}\label{thm5.1-15}
&\angle (x^{*} p^{*} \mu(t_{i})^{*}) 
+ \angle (p^{*} x^{*} \mu(t_{i})^{*}) 
+ \angle (p^{*} \mu(t_{i})^{*} x^{*}) - \pi \notag\\[3mm]
&=
\int_{\triangle (p^{*} x^{*} \mu(t_{i})^{*})} K^{*}\circ t \, dM^{*} \notag\\[3mm] 
&\ge
\frac{\angle (x^{*} p^{*} \mu(t_{i})^{*}) }{2\pi} 
\int_{M^{*}} K^{*}\circ t \, dM^{*} \notag\\[3mm]
&= 
\frac{\angle (x^{*} p^{*} \mu(t_{i})^{*}) }{2\pi} \, c(M^{*}).
\end{align}
Moreover, by (\ref{thm5.1-14}), we have
\begin{align}\label{thm5.1-16}
&\angle (x^{*} p^{*} \mu(t_{i})^{*}) 
+ \angle (p^{*} \mu(t_{i})^{*} x^{*})
- \pi / 2 \notag \\[2mm]
&\ge 
\angle (x^{*} p^{*} \mu(t_{i})^{*}) 
+ \angle (p^{*} x^{*} \mu(t_{i})^{*}) 
+ \angle (p^{*} \mu(t_{i})^{*} x^{*}) - \pi
\end{align}
Combining (\ref{thm5.1-15}) and (\ref{thm5.1-16}), we see 
\begin{equation}\label{thm5.1-17}
\angle (x^{*} p^{*} \mu(t_{i})^{*}) 
\ge 
\frac{\pi (\pi - 2 \, \angle (p^{*} \mu(t_{i})^{*} x^{*}) )}{2 \pi -c(M^{*})}.
\end{equation}
Since $K^{*} (t) \le 0$ on $[0, \infty)$ and (\ref{prop2009-09-29-1}), 
it follows from Lemma \ref{lem3.1} that 
\[
1 
\le 
\lim_{t \to \infty} m'(t) 
\le 
\exp \left( \int^{\infty}_{0} (-t \,K^{*} (t)) \,dt \right) 
< \infty.
\]
Thus, by the isoperimetric inequality (cf.\,\cite[Theorem 5.2.1]{SST}), 
we have 
\begin{equation}\label{thm5.1-18}
2 \pi -c(M^{*}) 
= 2 \pi \lim_{t \to \infty} m'(t) 
\le 2 \pi  \exp \left( \int^{\infty}_{0} (-t \,K^{*} (t)) \,dt \right) < \infty.
\end{equation}
Combining (\ref{2009-03-12-1}), (\ref{thm5.1-17}), and (\ref{thm5.1-18}), we have 
\begin{equation}\label{thm5.1-19}
\angle (\gamma'(0), \mu'(0)) 
\ge 
\left(
\frac{\pi}{2} - \angle (p^{*} \mu(t_{i})^{*} x^{*} )
\right)
\exp \left( \int^{\infty}_{0} t \,K^{*} (t) \,dt \right).
\end{equation}
Since (\ref{2009-03-12-3}) holds, we obtain, by taking the limit of $i$, 
\begin{equation}\label{thm5.1-2009-10-08-1}
\angle (\gamma'(0), \mu'(0)) 
\ge \delta(K^{*}).
\end{equation}
Since $\mu$ is arbitrarily taken, (\ref{thm5.1-2009-10-08-1}) implies that 
\begin{equation}\label{thm5.1-2009-10-08-2}
A_{p} \subset \ol{\B_{\pi - \delta (K^{*})} (-\gamma'(0))}
\end{equation}
for all minimal geodesic segments $\gamma$ joining $p$ to $x$. 
Here $-\gamma'(0)$ denotes the antipodal point of $\gamma'(0)$ in $\Sph^{n -1}_{p}$. 
Since $x$ is a critical point of $d (p, \, \cdot \,)$, 
there exist at least two minimal geodesic segments joining $p$ to $x$. 
Hence, it follows from (\ref{thm5.1-2009-10-08-2}) that 
there exists two distinct vectors $v_{1}, v_{2} \in \Sph^{n -1}_{p}$ such that 
$
A_{p} \subset \ol{\B_{\pi - \delta (K^{*})} (v_{1})} \cap \ol{\B_{\pi - \delta (K^{*})} (v_{2})}.
$
In particular,
$\vol_{\Sph^{n -1}_{p}} A_{p}
< \vol \B_{\pi - \delta (K^{*})} (v_{1}) = \vol \B_{\pi - \delta (K^{*})} (v_{2})
$.
This contradicts Lemma \ref{lem2009-08-29-2}. 
$\qedd$
\end{proof}

\begin{corollary}\label{cor3.6-2009-11-05}
Let $(M,p)$ be a complete non-compact Riemannian $n$-manifold $M$ 
whose radial sectional curvature at the base point $p$ is bounded from below by
the radial curvature function $G$ of a non-compact 
model surface of revolution $(\wt{M}, \tilde{p})$. If 
\[
\lim_{t \to \infty} \frac{\vol B_{t} (p)}{\vol B_{t}^{n} (\tilde{p})} 
\ge 1 - F(\delta (G_{-}))
\]
then $M$ is diffeomorphic to Euclidean $n$-space $\R^{n}$.
Here we denote by $B_{t}^{n} (\tilde{p})$ the open distance ball at 
$\tilde{p} \in \wt{M}^n$ with radius $t > 0$ in an $n$-dimensional model $(\wt{M}^n, \tilde{p})$ 
of $(\wt{M}, \tilde{p})$, and we set $G_{-} := \min \{0, G\}$.
\end{corollary}

\begin{proof}
By Theorem \ref{prop2009-09-29}, it is sufficient to prove the corollary in the case where 

\begin{equation}\label{cor3.6-2009-11-05-1}
\lim_{t \to \infty} \vol B_{t}^{n} (\tilde{p}) < \infty.
\end{equation}
Then, by (\ref{cor3.6-2009-11-05-1})
\[
\int_{0}^{\infty} f(t)^{n - 1}\,dt < \infty
\]
holds, where $f$ denotes the warping function of $\wt{M}$. 
Hence, we have 
$\liminf_{t \to \infty} f(t) = 0$.
Therefore, it follows from \cite[Theorem 1.2]{ST} that 
$M$ is diffeomorphic to $\R^{n}$.
$\qedd$
\end{proof}

\bigskip

\begin{center}
Kei KONDO $\cdot$ Minoru TANAKA 

\bigskip
Department of Mathematics\\
Tokai University\\
Hiratsuka City, Kanagawa Pref.\\ 
259\,--\,1292 Japan

\bigskip

{\small
$\bullet$\,our e-mail addresses\,$\bullet$

\bigskip 
\textit{e-mail of Kondo}:

\medskip
{\tt keikondo@keyaki.cc.u-tokai.ac.jp}

\medskip
\textit{e-mail of Tanaka}:

\medskip
{\tt tanaka@tokai-u.jp}
}
\end{center}

\end{document}